\documentclass[11pt, reqno]{amsart}
\usepackage{amssymb,latexsym,amsmath,amsfonts}
\usepackage{latexsym}
\usepackage[mathscr]{eucal}
\usepackage{bm}

\title[Decomposition of a $\Gamma_n$-contraction]
{Canonical decomposition of operators associated with the
symmetrized polydisc}

\author{Sourav Pal}
\address{Department of Mathematics, Indian Institute of Technology Bombay, Mumbai - 400076, India.}
\email{sourav@math.iitb.ac.in,
souravmaths@gmail.com}
\thanks{The author is supported by Seed Grant of IIT Bombay, CPDA and the INSPIRE Faculty Award
(Award No. DST/INSPIRE/04/2014/001462) of DST, India.}



\def ~{\hspace{1mm}}

\def\textmatrix#1&#2\\#3&#4\\{\bigl({#1 \atop #3}\ {#2 \atop #4}\bigr)}
\def\dispmatrix#1&#2\\#3&#4\\{\left({#1 \atop #3}\ {#2 \atop #4}\right)}
\newcommand{\beg}{\begin{equation}}
\newcommand{\eeg}{\end{equation}}
\newcommand{\ben}{\begin{eqnarray*}}
\newcommand{\een}{\end{eqnarray*}}

\newtheorem{thm}{Theorem}[section]

\newtheorem{lem}[thm]{Lemma}

\newtheorem{prop}[thm]{Proposition}
\numberwithin{equation}{section}
\theoremstyle{definition}
\newtheorem{defn}[thm]{Definition}

\def\textmatrix#1&#2\\#3&#4\\{\bigl({#1 \atop #3}\ {#2 \atop #4}\bigr)}
\def\dispmatrix#1&#2\\#3&#4\\{\left({#1 \atop #3}\ {#2 \atop #4}\right)}

\begin{document}

\begin{abstract}
A tuple of commuting operators $(S_1,\dots,S_{n-1},P)$ for which
the closed symmetrized polydisc $\Gamma_n$ is a spectral set is
called a $\Gamma_n$-contraction. We show that every
$\Gamma_n$-contraction admits a decomposition into a
$\Gamma_n$-unitary and a completely non-unitary
$\Gamma_n$-contraction. This decomposition is an analogue to the
canonical decomposition of a contraction into a unitary and a
completely non-unitary contraction. We also find new
characterizations for the set $\Gamma_n$ and
$\Gamma_n$-contractions.
\end{abstract}

\maketitle

\section{Introduction}

\noindent The open and closed \textit{symmetrized polydisc} (or,
\textit{symmetrized $n$-disc}) for $n\geq 2$ are the following
sets
\begin{align*}
\mathbb G_n &=\left\{ \left(\sum_{1\leq i\leq n} z_i,\sum_{1\leq i<j\leq n}z_iz_j,\dots,
\prod_{i=1}^n z_i \right): \,|z_i|< 1, i=1,\dots,n \right \}, \\
\Gamma_n & =\left\{ \left(\sum_{1\leq i\leq n} z_i,\sum_{1\leq
i<j\leq n}z_iz_j,\dots, \prod_{i=1}^n z_i \right): \,|z_i|\leq 1,
i=1,\dots,n \right \}.
\end{align*}
The symmetrized polydisc is a polynomially convex domain which has
attracted considerable attention in past two decades because of
its connection with the difficult problem of $\mu$-synthesis that
arises in $H^{\infty}$ approach of robust control (e.g, see
\cite{ALY13}). Apart from its rich function theory and complex
geometry, this domain has been extensively studied by the operator
theorists, \cite{ay-jfa, ay-jot, tirtha-sourav, BSR, sourav,
sourav2, pal-shalit}.\\

In this article, we study a commuting tuple of operators
$(S_1,\dots,S_{n-1},P)$ defined on a Hilbert space $\mathcal H$
for which $\Gamma_n$ is a spectral set (which we define in Section
2). Such an operator tuple is called a
$\Gamma_n$-\textit{contraction}. An appealing and convenient way
of describing an operator $n$-tuple is by an underlying compact
subset of $\mathbb C^n$ that is a spectral set for the tuple. In
1951, von Neumann introduced the notion of spectral set for
operators and this geometric approach towards understanding
operators succeeded when he described all contractions, that is
operators with norm not greater than $1$, as operators having the
closed unit disc $\overline{\mathbb D}$ of the complex plane as a
spectral set, \cite{vN}.\\

One of the landmark discoveries in one variable operator theory is
the canonical decomposition of a contraction which asserts that
every contraction operator admits a unique decomposition into two
orthogonal parts of which one is a unitary and the other is a
completely non-unitary contraction. So in geometric language, for
an operator $T$ acting on a Hilbert space $\mathcal H$ and having
$\overline{\mathbb D}$ as a spectral set, there exist unique
reducing subspaces $\mathcal H_1, \mathcal H_2$ of $T$ such that
$\mathcal H=\mathcal H_1 \oplus \mathcal H_2$, $T|_{\mathcal H_1}$
is a unitary that lives on the unit circle $\mathbb T$ and
$T|_{\mathcal H_2}$ is a completely non-unitary contraction (see
Theorem 3.2 in Ch-I, \cite{nagy} for details). Needless to mention
that a unitary is a normal operator for which the boundary
$\mathbb T$ of $\overline{\mathbb D}$ is a spectral set. Also a
contraction on a Hilbert space is said to be \textit{completely
non-unitary} if there is no reducing subspace on which the
operator acts as a unitary. There are natural analogues of unitary
and completely non-unitary contractions in the literature of
$\Gamma_n$-contraction, \cite{BSR}. A $\Gamma_n$-\textit{unitary}
is a commuting tuple of normal operators $(S_1,\dots,S_{n-1},P)$
for which the distinguished boundary $b\Gamma_n$ of $\Gamma_n$ is
a spectral set, where
\[
b\Gamma_n=\{ (s_1,\dots,s_{n-1},p)\in\Gamma_n\,:\, |p|=1 \}.
\]
A $\Gamma_n$-contraction $(S_1,\dots,S_{n-1},P)$ is said to be
\textit{completely non-unitary} if there is no joint reducing
subspace of $S_1,\dots,S_{n-1},P$ on which $(S_1,\dots,S_{n-1},P)$
acts as a $\Gamma_n$-unitary.\\

Since an operator having $\overline{\mathbb D}$ as a spectral set
admits a canonical decomposition, it is naturally asked whether
one can decompose operators having a particular domain in $\mathbb
C^n$ as a spectral set. In \cite{ay-jot}, Agler and Young answered
this question by showing an explicit orthogonal decomposition of a
$\Gamma_2$-contraction (Theorem 2.8, \cite{ay-jot}). The aim of
this article is to generalize the results in $n$ variables and
find an orthogonal decomposition for a $\Gamma_n$-contraction that
splits a $\Gamma_n$-contraction into two parts of which one is a
$\Gamma_n$-unitary and the other is a completely non-unitary
$\Gamma_n$-contraction. Therefore, the main result of this paper
is the following:
\begin{thm}\label{mainthm}
Let $(S_1,\dots,S_{n-1},P)$ be a $\Gamma_n$-contraction on a
Hilbert space $\mathcal H$. Let $\mathcal H_1$ be the maximal
subspace of $\mathcal H$ which reduces $P$ and on which $P$ is
unitary. Let $\mathcal H_2=\mathcal H\ominus \mathcal H_1$. Then
\begin{enumerate}
\item $\mathcal H_1,\mathcal H_2$ reduce $S_1,\dots, S_{n-1}$,
\item $(S_1|_{\mathcal H_1},\dots,S_{n-1}|_{\mathcal
H_1},P|_{\mathcal H_1})$ is a $\Gamma_n$-unitary, \item
$(S_1|_{\mathcal H_2},\dots,S_{n-1}|_{\mathcal H_2},P|_{\mathcal
H_2})$ is a completely non-unitary $\Gamma_n$-contraction.
\end{enumerate}
The subspaces $\mathcal H_1$ or $\mathcal H_2$ may equal to the
trivial subspace $\{0\}$.
\end{thm}
This is Theorem \ref{thm:decomp} in this paper. We shall define
operator functions $\Phi_1,\dots,\Phi_{n-1}$ related to a
$\Gamma_n$-contraction in Section 2 which play central role in
this decomposition. We shall see that such decomposition is
possible because $\Phi_1,\dots,\Phi_{n-1}$ are positive
semi-definite. Also, we find few properties of the set $\Gamma_n$
and characterize the $\Gamma_n$-contractions in different ways. We
accumulate these results in Section 2. Also we provide a brief
background material in Section 2.

\section{Background material and preparatory results}

A compact subset $X$ of $\mathbb C^n$ is said to be a
\textit{spectral set} for a commuting $n$-tuple of bounded
operators $\underline{T}=(T_1,\hdots,T_n)$ defined on a Hilbert
space $\mathcal H$ if the Taylor joint spectrum
$\sigma_T(\underline{T})$ of $\underline{T}$ is a subset of $X$
and
\[
\|f(\underline{T})\|\leq \|f\|_{\infty,
X}=\sup\{|f(z_1,\hdots,z_n)|\,:\,(z_1,\hdots,z_n)\in\ X\}\,,
\]
for all rational functions $f$ in $\mathcal R(X)$. Here $\mathcal
R(X)$ denotes the algebra of all rational functions on $X$, that
is, all quotients $p/q$ of holomorphic polynomials $p,q$ in
$n$-variables for which $q$ has no zeros in $X$.\\

For $n\geq 2$, the symmetrization map in $n$-complex variables
$z=(z_1,\dots,z_n)$ is the following proper holomorphic map
\[
\pi_n(z)=(s_1(z),\dots, s_{n-1}(z), p(z))
\]
 where
 \[
s_i(z)= \sum_{1\leq k_1 \leq k_2 \dots \leq k_i \leq n}
z_{k_1}\dots z_{k_i} \;,\; i=1,\dots,n-1 \quad \text{ and } \quad
p(z)=\prod_{i=1}^{n}z_i\,.
 \]
The closed \textit{symmetrized} $n$-\textit{disk} (or simply
closed \textit{symmetrized polydisc}) is the image of the closed
unit $n$-disc $\overline{\mathbb D^n}$ under the symmetrization
map $\pi_n$, that is, $\Gamma_n := \pi_n(\overline{\mathbb D^n})$.
Similarly the open symmetrized polydisc $\mathbb G_n$ is defined
as the image of the open unit polydisc $\mathbb D^n$ under
$\pi_n$. For simplicity we write down explicitly the set
$\Gamma_n$ for $n=2$ and $3$.
\begin{align*}
\Gamma_2 &=\{ (z_1+z_2,z_1z_2): \,|z_i|\leq 1, i=1,2 \} \\
\Gamma_3 & =\{ (z_1+z_2+z_3,z_1z_2+z_2z_3+z_3z_1,z_1z_2z_3):
\,|z_i|\leq 1, i=1,2,3 \}.
\end{align*}
The set $\Gamma_n$ is polynomially convex but not convex (see
\cite{BSR}). We obtain from the literature (see \cite{BSR}) the
fact that the distinguished boundary of the symmetrized polydisc
is the symmetrization of the distinguished boundary of the
$n$-dimensional polydisc, which is $n$-torus $\mathbb T^n$. Hence
the distinguished boundary $b\Gamma_n$ of $\Gamma_n$ is the set
\begin{align*}
b\Gamma_n =\left\{ \left(\sum_{1\leq i\leq n} z_i,\sum_{1\leq
i<j\leq n}z_iz_j,\dots, \prod_{i=1}^n z_i \right): \,|z_i|= 1,
i=1,\dots,n \right \}.
\end{align*}

Operator theory on the symmetrized polydiscs of dimension $2$ and
$n$ have been extensively studied in past two decades
\cite{ay-jfa, ay-jot, tirtha-sourav, BSR, sourav, pal-shalit}.

\begin{defn}
A tuple of commuting operators $(S_1,\dots,S_{n-1},P)$ on a
Hilbert space $\mathcal H$ for which $\Gamma_n$ is a spectral set
is called a $\Gamma_n$-$contraction$.
\end{defn}

It is evident from the definition that if $(S_1,\dots,S_{n-1},P)$
is a $\Gamma_n$-contraction then the $S_i$ have norm not greater
than $n$ and $P$ is a contraction. Unitaries, isometries and
co-isometries are important special classes of contractions. There
are natural analogues of these classes for
$\Gamma_n$-contractions.

\begin{defn}
Let $S_1,\dots,S_{n-1},P$ be commuting operators on a Hilbert
space $\mathcal H$. We say that $(S_1,\dots,S_{n-1},P)$ is
\begin{itemize}
\item [(i)] a $\Gamma_n$-\textit{unitary} if $S_1,\dots,S_{n-1},P$
are normal operators and the Taylor joint spectrum
$\sigma_T(S_1,\dots,S_{n-1},P)$ is contained in $b\Gamma_n$ ;
\item [(ii)] a $\Gamma_n$-\textit{isometry} if there exists a
Hilbert space $\mathcal K$ containing $\mathcal H$ and a
$\Gamma_n$-unitary $(\tilde{S_1},\dots,\tilde{S_{n-1}},\tilde{P})$
on $\mathcal K$ such that $\mathcal H$ is a common invariant
subspace for $\tilde{S_1},\dots,\tilde{S_{n-1}},\tilde{P}$ and
that $S_i=\tilde{S_i}|_{\mathcal H}$ for $i=1,\dots,n-1$ and
$\tilde{P}|_{\mathcal H}=P$; \item [(iii)] a
$\Gamma_n$-\textit{co-isometry} if $(S_1^*,\dots,S_{n-1}^*,P^*)$
is a $\Gamma_n$-isometry; \item [(iv)] a \textit{completely
non-unitary} $\Gamma_n$-\textit{contraction} if $P$ is a
completely non-unitary contraction.
\end{itemize}
\end{defn}

One can easily verify that if $(S_1,\dots,S_{n-1},P)$ is a
completely non-unitary $\Gamma_n$-contraction on a Hilbert space
$\mathcal H$, then there is no non-trivial subspace of $\mathcal
H$ that reduces $S_1,\dots,S_{n-1},P$ and on which
$(S_1,\dots,S_{n-1},P)$ acts as a $\Gamma_n$-unitary.\\

For a $\Gamma_n$-contraction $(S_1,\dots,S_{n-1},P)$, let us
define $n-1$ operator pencils $\Phi_1,\dots,\Phi_{n-1}$ in the
following way. These operator functions will play central role in
the canonical decomposition of $(S_1,\dots,S_{n-1},P)$.
\begin{align}
\Phi_{i}(S_1,\dots, S_{n-1},P) &=
(n-S_i)^*(n-S_i)-(nP-S_{n-i})^*(nP-S_{n-i}) \notag
\\&
=n^2(I-P^*P)+(S_i^*S_i-S_{n-i}^*S_{n-i})-n(S_i-S_{n-i}^*P) \notag \\
& \quad \quad -n(S_i^*-P^*S_{n-i}) \label{eq:1a}.
\end{align}
So in particular when $S_1,\dots,S_{n-1}, P$ are scalars, i.e,
points in $\Gamma_n$, the above operator pencils take the
following form for $i=1,\dots, n-1$:
\begin{equation}
\Phi_{i}(s_1,\dots,s_{n-1},p) =
n^2(1-|p|^2)+(|s_i|^2-|s_{n-i}|^2)-n(s_i-\bar{s}_{n-i}p)-n(\bar{s}_i-\bar{p}s_{n-i}).
\label{eqn:2a}
\end{equation}

\begin{lem}\label{thm:sc1}

Let $(s_1,\dots, s_{n-1},p)\in \mathbb C^n$. Then the following
are equivalent:
\begin{enumerate}
\item $(s_1,\dots, s_{n-1},p)\in\Gamma_n$\,; \item $(\omega
s_1,\dots,\omega^{n-1} s_{n-1}, \omega^n p)\in \Gamma_n$ for all
$\omega\in\mathbb T$ \,;  \item $|p|\leq 1$ and there exists
$(c_1,\dots,c_{n-1})\in \Gamma_{n-1}$ such that

\[
s_i=c_i+\bar{c}_{n-i}p \text{ for } i=1,\dots,n-1.
\]
\end{enumerate}

\end{lem}

\begin{proof}
(1)$\Leftrightarrow (3)$ has been established in \cite{costara1}
(see Theorem 3.7 in \cite{costara1} for a proof). We prove here
(1)$\Leftrightarrow (2)$.  Let $(s_1,\dots, s_{n-1},p)\in
\Gamma_n$. Then there are points $z_1,\dots,z_n$ in
$\overline{\mathbb D}$ such that
\[
(s_1,\dots, s_{n-1},p)=\pi_n (z_1,\dots,z_n).
\]
Now for any $\omega \in\mathbb T$, consider the points $\omega
z_1, \dots, \omega z_n$ in $\overline{\mathbb D}$. Clearly
\[
(\omega s_1, \dots, \omega^{n-1}s_{n-1}, \omega^n p)=\pi_n (\omega
z_1, \dots, \omega z_n).
\]
Therefore, $(\omega s_1, \dots,
\omega^{n-1}s_{n-1}s_{n-1},\omega^n p ) \in \Gamma_n$. The other
side of the proof is trivial.

\end{proof}

In a similar fashion, we have the following characterizations for
$\Gamma_n$-contractions.

\begin{thm}\label{lem:3}
Let $(S_1,\dots,S_{n-1},P)$ be a tuple of commuting operators
acting on a Hilbert space $\mathcal H$. Then the following are
equivalent:

\begin{enumerate}
\item $(S_1,\dots,S_{n-1},P)$ is a $\Gamma_n$-contraction\,; \item
for all holomorphic polynomials $f$ in $n$-variables
\[
\|f(S_1,\dots,S_{n-1},P)\|\leq \|f\|_{\infty,\Gamma_n}\,;
\]
\item $(\omega S_1,\dots,\omega^{n-1} S_{n-1},\omega^n P)$ is a
$\Gamma_n$-contraction for any $\omega\in\mathbb T$.
\end{enumerate}

\end{thm}

\begin{proof}

$(1)\Rightarrow (2)$ follows from definition of spectral set and
$(2)\Rightarrow (1)$ just requires polynomial convexity of the set
$\Gamma_3$. We prove here $(1)\Rightarrow (3)$ because
$(3)\Rightarrow (1)$ is obvious. Let $f(s_1,\dots,s_{n-1},p)$ be a
holomorphic polynomial in the co-ordinates of $\Gamma_n$ and for
$\omega\in\mathbb T$ let $ f_1(s_1,\dots,s_{n-1},p)=f(\omega
s_1,\dots,\omega^{n-1} s_{n-1},\omega^n p)$. It is evident from
part $(1)\Rightarrow (2)$ that
\begin{align*}
& \sup\{|f(\omega s_1,\dots,\omega^{n-1}s_{n-1},\omega^n
p)|\,:\,(s_1,\dots,s_{n-1},p)\in\Gamma_n \}\\&
=\sup\{|f_1(s_1,\dots,s_{n-1},p)|\,:\,(s_1,\dots,s_{n-1},p)\in
\Gamma_n\}.
\end{align*}
Therefore,
\begin{align*}
\|f(\omega S_1,\dots, \omega^{n-1} S_{n-1}, \omega^n P)\|&
=\|f_1(S_1,\dots,S_{n-1},P)\|
\\& \leq \|f_1\|_{\infty, \Gamma_n} \\& =\|f\|_{\infty,\Gamma_n}.
\end{align*}
Therefore, by $(1)\Rightarrow (2)$, $(\omega S_1,\dots,
\omega^{n-1} S_{n-1}, \omega^n P)$ is a $\Gamma_n$-contraction.
\end{proof}

\begin{prop}\label{prop:sc1}
Let $(s_1,\dots,s_{n-1},p)\in\mathbb C^n$. Then in the following
$(1)\Rightarrow (2) \Leftrightarrow (3)$.
\begin{enumerate}
\item $(s_1,\dots,s_{n-1},p)\in\Gamma_n$; \item for
$i=1,\dots,n-1$, $\Phi_i(\alpha s_1,\dots,
\alpha^{n-1}s_{n-1},\alpha^n p)\geq 0$ for all $\alpha
\in\overline{\mathbb D}$; \item for $i=1,\dots,n-1$, $|n \alpha^n
p - \alpha^{n-i}s_{n-i}|\leq |n-\alpha^is_i|$ for all $\alpha
\in\overline{\mathbb D}$.
\end{enumerate}
\end{prop}

\begin{proof}
\noindent $(1)\Rightarrow (2).$  Let $(s_1,\dots,s_{n-1},p)\in
\Gamma_n$ and let $\alpha\in \mathbb D$. Then
\[
(\alpha s_1,\dots,{\alpha}^{n-1}s_{n-1},{\alpha}^np)\in\mathbb
G_n.
\]
We apply Lemma \ref{thm:sc1} and get
$(c_1,\dots,c_{n-1})\in\Gamma_{n-1}$ such that
\[
\alpha^i s_i=c_i+\bar{c}_{n-i}(\alpha^n p) \text{ for }
i=1,\dots,n-1.
\]
We shall use the following notations here:
\begin{align*}
a &= 1-|\alpha^np|^2 \\
m & = |\alpha^i s_i|^2-|\alpha^{n-i} s_{n-i}|^2 \\
b & = \alpha^i s_i- \overline{\alpha^{n-i} s_{n-i}}(\alpha^n p) \\
c & ={\alpha}^{n-i}{s_{n-i}}-\overline{\alpha^i s_i}(\alpha^n p).
\end{align*}

We first show that
\begin{align}\label{eqn:0001}
n^2a+m & \geq 2n|b| \\
n^2a-m & \geq 2n|c|.
\end{align}

Now
\begin{align*}
m &= |\alpha^i s_i|^2-|\alpha^{n-i} s_{n-i}|^2 \\
&= |c_i+\bar{c}_{n-i}(\alpha^n p)|^2-|c_{n-i}+\bar{c}_{i}(\alpha^np)|^2 \\
&=(|c_i|^2+|c_{n-i}(\alpha^n p)|^2+c_ic_{n-i}(\overline{\alpha^n p})+\bar{c}_i\bar{c}_{n-i}(\alpha^n p)) \\
&\quad - (|c_{n-i}|^2+|c_i(\alpha^n
p)|^2+c_ic_{n-i}(\overline{\alpha^n
p})+\bar{c}_i\bar{c}_{n-i}(\alpha^n p))\\
&= (|c_i|^2-|c_{n-i}|^2)(1-|\alpha^n p|^2)\\
&= (|c_i|^2-|c_{n-i}|^2)a \,.
\end{align*}
Also
\begin{align*}
|b| &=|\alpha^i s_i -\overline{(\alpha^{n-i} s_{n-i})}(\alpha^n p)| \\
& = |(c_i+\bar c_{n-i} {\alpha^n p}) - (\bar c_{n-i} + c_i
\overline{\alpha ^n p})\alpha^n p| \\
& = |c_i|(1- |\alpha^n p|^2)\\
& = |c_i|a\,,
\end{align*}
and
\begin{align*}
|c| &=|\alpha^{n-i} s_{n-i} -\overline{(\alpha^i s_i)}(\alpha^n p)| \\
& = |(c_{n-i}+\bar c_i {\alpha^n p}) - (\bar c_i + c_{n-i}
\overline{\alpha ^n p})\alpha^n p| \\
& = |c_{n-i}|(1- |\alpha^n p|^2)\\
& = |c_{n-i}|a.
\end{align*}

Therefore,
\begin{align*}
n^2a+m-2n|b| &= n^2a+(|c_i|^2-|c_{n-i}|^2)a - 2n|c_i|a \\
&= \{ (n-|c_i|)^2-|c_{n-i}|^2 \}a \\
& = (n-|c_i|+|c_{n-i}|)(n-|c_i|-|c_{n-i}|)a \\
& \geq 0.
\end{align*}
The last inequality follows from the facts that $a>0$ and that
$|c_i|+|c_{n-i}| \leq n$ as $s_i\in\Gamma_n$. Therefore,
\[
n^2a+m\geq 2n|b|.
\]
Now using the fact that
\begin{equation}\label{eq:4}
x>|y|\Leftrightarrow x> \text{Re } \omega y\,, \quad \text{ for
all }\omega\in\mathbb T\,,
\end{equation}
we have that

\begin{align}
n^2a+m & \geq 2n \text{ Re } \omega b \notag \\
& \notag = n \omega b +n \bar{\omega}\bar b, \quad \text{ for all
} \omega\in\mathbb T.
\end{align}
Choosing $\omega=1$ and substituting the values of $a,m,b$ we get
\begin{align*}
\Phi_{i}(\alpha s_1,\dots,\alpha^{n-i} s_{n-i}, \alpha^n p) &=
n^2(1-|\alpha^n
p|^2)+(|\alpha^i s_i|^2-|\alpha^{n-i} s_{n-i}|^2)\\
& \quad -n(\alpha^i s_i-\overline{(\alpha^{n-i} s_{n-i})}(\alpha^n
p))\\
& \quad -n\overline{(\alpha^i s_i-\overline{(\alpha^{n-i}
s_{n-i})}(\alpha^n p))}
\\& \;
\geq 0 .
\end{align*}
By continuity, we have that $\Phi_{i}(\alpha s_1,\dots,
\alpha^{n-i} s_{n-i},\alpha^n p)\geq 0$, for all $\alpha
\in\overline{\mathbb D}$.\\

\noindent $(2)\Leftrightarrow (3).$ From (\ref{eq:1a}) we have
that
\begin{align*}
& \Phi_{i}(\alpha s_1,\dots,\alpha^{n-1} s_{n-1},\alpha^n p) \\ &
=(n-\bar{\alpha}^i \bar{s}_i)(n-\alpha^i s_i)-(n\bar{\alpha}^n\bar
p-\bar{\alpha}^{n-i}\bar{s}_{n-i})(n\alpha^n p-\alpha^{n-i}
s_{n-i}).
\end{align*}
Therefore,
\begin{align*}
&\Phi_{i}(\alpha^i s_i,\dots,\alpha^{n-1} s_{n-1},\alpha^n p) \geq 0 \\
\Leftrightarrow &(n-\bar{\alpha}^i \bar{s}_i)(n-\alpha^i
s_i)-(n\bar{\alpha}^n\bar
p-\bar{\alpha}^{n-i}\bar{s_{n-i}})(n\alpha^n
p-\alpha^{n-i} s_{n-i}) \geq 0 \\
\Leftrightarrow &\left | {n\alpha^n p-\alpha^{n-i} s_{n-i}} \right
| \leq |{n-\alpha s_i}|\,.
\end{align*}

\end{proof}

\begin{prop}\label{lem:3}
Let $(S_1,\dots,S_{n-1},P)$ be a $\Gamma_n$-contraction. Then for
$i=1,\dots,n-1$ $\Phi_i(\alpha
S_1,\dots,\alpha^{n-1}S_{n-1},\alpha^n P)\geq 0$ for all $\alpha
\in\overline{\mathbb D}$.
\end{prop}

\begin{proof}
Since $(S_1,\dots,S_{n-1},P)$ is a $\Gamma_n$-contraction,
$\sigma_T(S_1,\dots,S_{n-1},P)\subseteq \Gamma_n$. Let $f$ be a
holomorphic function in a neighbourhood of $\Gamma_n$. Since
$\Gamma_n$ is polynomially convex, by Oka-Weil Theorem (see
\cite{Gamelin}, Theorem 5.1) there is a sequence of polynomials
$\{p_k\}$ in $n$-variables such that $p_k\rightarrow f$ uniformly
over $\Gamma_n$. Therefore, by Theorem 9.9 of CH-III in
\cite{vasilescu},
\[
p_k(S_1,\dots,S_{n-1},P)\rightarrow f(S_1,\dots,S_{n-1},P)
\]
which by the virtue of $(S_1,\dots,S_{n-1},P)$ being a
$\Gamma_n$-contraction implies that
\[
\| f(S_1,\dots,S_{n-1},P) \|=\displaystyle \lim_{k\rightarrow
\infty}\| p_k(S_1,\dots,S_{n-1},P) \|\leq \displaystyle
\lim_{k\rightarrow
\infty}\|p_k\|_{\infty,\Gamma_n}=\|f\|_{\infty,\Gamma_n}.
\]
We fix $\alpha \in \mathbb D$ and choose
\[
f(s_1,\dots,s_{n-1},p)=\frac{n\alpha^np-\alpha^{n-i}s_{n-i}}{n-\alpha^i
s_i}\,.
\]
It is evident that $f$ is well-defined and is holomorphic in a
neighborhood of $\Gamma_n$ and has norm not greater than $1$, by
part-(3) of Proposition \ref{prop:sc1}. So we get
\[
\|(n\alpha^nP-\alpha^{n-i}S_{n-i})(n-\alpha^i S_i)^{-1} \|\leq
\|f\|_{\infty,\Gamma_n}\leq 1.
\]
Thus
\[
{(n-\alpha^i
S_i)^*}^{-1}(n\alpha^nP-\alpha^{n-i}S_{n-i})^*(n\alpha^nP-\alpha^{n-i}S_{n-i})(n-\alpha^i
S_i)^{-1}\leq I
\]
which is equivalent to
\[
(n-\alpha^i S_i)^*(n-\alpha^i S_i)\geq
(n\alpha^nP-\alpha^{n-i}S_{n-i})^*(n\alpha^nP-\alpha^{n-i}S_{n-i}).
\]
By the definition of $\Phi_{i}$, this is same as saying that
\[
\Phi_{i}(\alpha S_1,\dots,\alpha^{n-1} S_{n-1},\alpha^n P)\geq 0
\text{ for all } \alpha \in \mathbb D.
\]
By continuity we have that
\[
\Phi_{i}(\alpha S_1,\dots,\alpha^{n-1} S_{n-1},\alpha^nP)\geq 0
\quad \text{ for all } \alpha\in \overline{\mathbb D}.
\]
\end{proof}

Here is a set of characterizations for the $\Gamma_n$-unitaries.

\begin{thm}[Theorem 4.2, \cite{BSR}]\label{thm:tu}
Let $(S_1,\dots, S_{n-1}, P)$ be a commuting triple of bounded
operators. Then the following are equivalent.

\begin{enumerate}

\item $(S_1,\dots,S_{n-1},P)$ is a $\Gamma_n$-unitary,

\item $P$ is a unitary and $(S_1,\dots,S_{n-1},P)$ is a
$\Gamma_n$-contraction,

\item $P$ is a unitary,
$(\frac{n-1}{n}S_1,\frac{n-2}{n}S_2,\dots,\frac{1}{n}S_{n-1})$ is
a $\Gamma_{n-1}$-contraction and $S_i = S_{n-i}^* P$ for
$i=1,\dots,n-1$.
\end{enumerate}
\end{thm}

\section{The orthogonal decomposition of a $\Gamma_n$-contraction}

We now state and prove the main result of this paper which we have
mentioned in the introduction as Theorem \ref{mainthm}.

\begin{thm}\label{thm:decomp}
Let $(S_1,\dots,S_{n-1},P)$ be a $\Gamma_n$-contraction on a
Hilbert space $\mathcal H$. Let $\mathcal H_1$ be the maximal
subspace of $\mathcal H$ which reduces $P$ and on which $P$ is
unitary. Let $\mathcal H_2=\mathcal H\ominus \mathcal H_1$. Then
\begin{enumerate}
\item $\mathcal H_1,\mathcal H_2$ reduce $S_1,\dots, S_{n-1}$,
\item $(S_1|_{\mathcal H_1},\dots,S_{n-1}|_{\mathcal
H_1},P|_{\mathcal H_1})$ is a $\Gamma_n$-unitary, \item
$(S_1|_{\mathcal H_2},\dots,S_{n-1}|_{\mathcal H_2},P|_{\mathcal
H_2})$ is a completely non-unitary $\Gamma_n$-contraction.
\end{enumerate}
The subspaces $\mathcal H_1$ or $\mathcal H_2$ may equal to the
trivial subspace $\{0\}$.
\end{thm}

\begin{proof}

First we consider the case when $P$ is a completely non-unitary
contraction. Then obviously $\mathcal H_1=\{0\}$ and if $P$ is a
unitary then $\mathcal H=\mathcal H_1$ and so $\mathcal
H_2=\{0\}$. In such cases the theorem is trivial. So let us
suppose that $P$ is neither a unitary nor a completely non unitary
contraction. With respect to the decomposition $\mathcal
H=\mathcal H_1\oplus \mathcal H_2$, let
\begin{align*}
S_1 & =
\begin{bmatrix}
S_{111} & S_{112}\\
S_{121} & S_{122}
\end{bmatrix}\,,\,
S_2=
\begin{bmatrix}
S_{211} & S_{212}\\
S_{221} & S_{222}
\end{bmatrix}
,\dots,
\begin{bmatrix}
S_{(n-1)11} & S_{(n-1)12}\\
S_{(n-1)21} & S_{(n-1)22}
\end{bmatrix}
\text{ and } \\
P & =
\begin{bmatrix}
P_1&0\\
0&P_2
\end{bmatrix},
\end{align*}
so that $P_1$ is a unitary and $P_2$ is completely non-unitary.
Since $P_2$ is completely non-unitary it follows that if
$h\in\mathcal H$ and
\[
\|P_2^nh\|=\|h\|=\|{P_2^*}^nh\|, \quad n=1,2,\hdots
\]
then $h=0$.\\

For an arbitrary $i$ between $1$ and $n-1$, the commutativity of
$S_i$ and $P$ gives us

\begin{align}
S_{i11}P_1&=P_1S_{i11}    & S_{i12}P_2=P_1S_{i12}\,, \label{eqn:1} \\
S_{i21}P_1&=P_2S_{i21}    & S_{i22}P_2=P_2S_{i22}\,. \label{eqn:2}
\end{align}

Also by the commutativity of $S_{n-i}$ and $P$ we obtain

\begin{align}
S_{(n-i)11}P_1&=P_1S_{(n-i)11}    & S_{(n-i)12}P_2=P_1S_{(n-i)12}\,, \label{eqn:3} \\
S_{(n-i)21}P_1&=P_2S_{(n-i)21}    &
S_{(n-i)22}P_2=P_2S_{(n-i)22}\,. \label{eqn:4}
\end{align}
By Proposition \ref{lem:3}, we have for all $\omega,
\beta\in\mathbb T$,
\begin{align*}
\Phi_i(\omega S_1,\dots,\omega^{n-1} S_{n-1},\omega^n
P) & =n^2(I-P^*P)+(S_i^*S_i-S_{n-i}^*S_{n-i})\\
& \quad \quad -2n\text{
Re }\omega^i(S_i-S_{n-i}^*P)\\
& \geq 0 \,,\\
\Phi_{n-i}(\beta S_1,\dots,\beta^{n-1} S_{n-1}, \beta^n P) &
=n^2(I-P^*P)+(S_{n-i}^*S_{n-i}-S_i^*S_i)\\
& \quad \quad -2n\text{ Re }\beta^{n-i}(S_{n-i}-S_i^*P)\\
& \geq 0 \,.
\end{align*}
Adding $\Phi_i$ and $\Phi_{n-i}$ we get
\[
n(I-P^*P)-\text{Re }\omega^i(S_i-S_{n-i}^*P)-\text{Re }
\beta^{n-i}(S_{n-i}-S_i^*P)\geq 0
\]
that is
\begin{align}\label{eqn:5}
\begin{bmatrix}
0&0\\
0&n(I-P_2^*P_2)
\end{bmatrix}
-& \text{ Re }\omega^i
\begin{bmatrix}
S_{i11}-S_{(n-i)11}^*P_1 & S_{i12}-S_{(n-i)21}^*P_2\\
S_{i21}-S_{(n-i)12}^*P_1&S_{i22}-S_{(n-i)22}^*P_2
\end{bmatrix} \\
-&\text{ Re }\beta^{n-i}
\begin{bmatrix}
S_{(n-i)11}-S_{i11}^*P_1&S_{(n-i)12}-S_{i21}^*P_2\\
S_{(n-i)21}-S_{i12}^*P_1&S_{(n-i)22}-S_{i22}^*P_2
\end{bmatrix}\, \geq 0 \notag
\end{align}
for all $\omega,\beta\in\mathbb T$. Since the matrix in the left
hand side of (\ref{eqn:5}) is self-adjoint, if we write
(\ref{eqn:5}) as

\begin{equation}\label{eqn:6}
\begin{bmatrix}
R&X\\
X^*&Q
\end{bmatrix}
\geq 0\,,
\end{equation}
then

\begin{eqnarray*}\begin{cases}
&(\mbox{i})\; R\,, Q \geq 0 \text{ and } R=-\text{ Re }\omega^i (
S_{i11}-S_{(n-i)11}^*P_1)\\& \quad \quad \quad \quad \quad \quad
\quad \quad \quad \quad -\text{ Re }\beta^{n-i}
(S_{(n-i)11}-S_{i11}^*P_1)\\& (\mbox{ii}) X= -\frac{1}{2} \{
\omega^i (
S_{i12}-S_{(n-i)21}^*P_2)+\bar{\omega^i}(S_{i21}^*-P_1^*S_{(n-i)12})\\&
\quad \quad \quad \quad + \beta^{n-i}
(S_{(n-i)12}-S_{i21}^*P_2)+\bar{\beta^{n-i}}(S_{(n-i)21}^*-P_1^*S_{i12})
\}
\\&(\mbox{iii})\; Q=3(I-P_2^*P_2)-\text{ Re }\omega^i (S_{i22}-S_{(n-i)22}^*P_2)\\
& \quad \quad \quad \quad - \text{ Re }\beta^{n-i}
(S_{(n-i)22}-S_{i22}^*P_2) \,.
\end{cases}
\end{eqnarray*}

Since the left hand side of (\ref{eqn:6}) is a positive
semi-definite matrix for every $\omega$ and $\beta$, if we choose
$\beta^{n-i}=1$ and $\beta^{n-i}=-1$ respectively then
consideration of the $(1,1)$ block of (\ref{eqn:5}) reveals that
\[
\omega^i(S_{i11}-S_{(n-i)11}^*P_1)+\bar{\omega}^i(S_{i11}^*-P_1^*S_{(n-i)11})\leq
0
\]
for all $\omega\in\mathbb T$. Choosing $\omega^i =\pm 1$ we get
\begin{equation}\label{eqn:7}
(S_{i11}-S_{(n-i)11}^*P_1)+(S_{i11}^*-P_1^*S_{(n-i)11})=0
\end{equation}
and choosing $\omega^i =\pm i$ we get
\begin{equation}\label{eqn:8}
(S_{i11}-S_{(n-i)11}^*P_1)-(S_{i11}^*-P_1^*S_{(n-i)11})=0\,.
\end{equation}
Therefore, from (\ref{eqn:7}) and (\ref{eqn:8}) we get
\[
S_{i11}=S_{(n-i)11}^*P_1\,,
\]
where $P_1$ is unitary. Similarly, we can show that
\[
S_{(n-i)11}=S_{i11}^*P_1\,.
\]
Therefore, $R=0$. Since $(S_1,\dots,S_{n-1},P)$ is a
$\Gamma_n$-contraction, $\|S_{n-i}\|\leq 3$ and hence
$\|S_{(n-i)11}\|\leq 3$. Also by Lemma 2.5 of \cite{BSR},
$(\frac{n-1}{n}S_1,\frac{n-2}{n}S_2,\dots,\frac{1}{n}S_{n-1})$ is
a $\Gamma_{n-1}$-contraction and hence
$(\frac{n-1}{n}S_{111},\frac{n-2}{n}S_{211},\dots,
\frac{1}{n}S_{(n-1)11})$ is a $\Gamma_{n-1}$-contraction.
Therefore, by part-(3) of Theorem
\ref{thm:tu}, $(S_{111},\dots,S_{(n-1)11},P_1)$ is a $\Gamma_n$-unitary.\\

Now we apply Proposition 1.3.2 of \cite{bhatia} to the positive
semi-definite matrix in the left hand side of (\ref{eqn:6}). This
Proposition states that if $R,Q \geq 0$ then $\begin{bmatrix} R&X
\\ X^*&Q
\end{bmatrix} \geq 0$ if and only if $X=R^{1/2}KQ^{1/2}$ for
some contraction $K$.\\

\noindent Since $R=0$, we have $X=0$. Therefore,
\begin{align*}
0=\; & \omega^i (
S_{i12}-S_{(n-i)21}^*P_2)+\bar{\omega}^i(S_{i21}^*-P_1^*S_{(n-i)12})\\&
+
\beta^{n-i}(S_{(n-i)12}-S_{i21}^*P_2)+\bar{\beta}^{n-i}(S_{(n-i)21}^*-P_1^*S_{i12})\;,
\end{align*}
for all $\omega,\beta \in\mathbb T$. Choosing $\beta^{n-i} =\pm 1$
we get
\[
\omega^i (
S_{i12}-S_{(n-i)21}^*P_2)+\bar{\omega}^i(S_{i21}^*-P_1^*S_{(n-i)12})=0\;,
\]
for all $\omega \in \mathbb T$. With the choices $\omega^i=1,i$,
this gives
\[
S_{i12}=S_{(n-i)21}^*P_2\,.
\]
Therefore, we also have
\[
S_{i21}^*=P_1^*S_{(n-i)12}\,.
\]
Similarly, we can prove that
\[
S_{(n-i)12}=S_{i21}^*P_2\,,\quad S_{(n-i)21}^*=P_1^*S_{i12}\,.
\]
Thus, we have the following equations
\begin{align}
& S_{i12}=S_{(n-i)21}^*P_2         & S_{i21}^*=P_1^*S_{(n-i)12} \label{eqn:9}\\
& S_{(n-i)12}=S_{i21}^*P_2         &S_{(n-i)21}^*=P_1^*S_{i12}\,.
\label{eqn:10}
\end{align}
Thus from (\ref{eqn:9}), $S_{i21}=S_{(n-i)12}^*P_1$ and together
with the first equation in (\ref{eqn:2}), this implies that
\[
S_{(n-i)12}^*P_1^2=S_{i21}P_1=P_2S_{i21}=P_2S_{(n-i)12}^*P_1
\]
and hence
\begin{equation}\label{eqn:11}
S_{(n-i)12}^*P_1=P_2S_{(n-i)12}^*\,.
\end{equation}
From equations in (\ref{eqn:3}) and (\ref{eqn:11}) we have that
\[
S_{(n-i)12}P_2=P_1S_{(n-i)12}\,, \quad
S_{(n-i)12}{P_2^*}={P_1^*}S_{(n-i)12}.
\]
Thus
\begin{align*}
S_{(n-i)12}P_2{P_2^*} &=P_1S_{(n-i)12}{P_2^*}
=P_1{P_1^*}S_{(n-i)12}
=S_{(n-i)12}\,, \\
S_{(n-i)12}{P_2^*}P_2 &= {P_1^*}S_{(n-i)12}P_2
={P_1^*}P_1S_{(n-i)12}=S_{(n-i)12}\,,
\end{align*}
and so we have
\[
P_2{P_2^*}S_{(n-i)12}^*=S_{(n-i)12}^*={P_2^*}P_2S_{(n-i)12}^*\,.
\]
This shows that $P_2$ is unitary on the range of $S_{(n-i)12}^*$
which can never happen because $P_2$ is completely non-unitary.
Therefore, we must have $S_{(n-i)12}^*=0$ and so $S_{(n-i)12}=0$.
Similarly we can prove that $S_{i12}=0$. Also from (\ref{eqn:9}),
$S_{i21}=0$ and from (\ref{eqn:10}), $S_{(n-i)21}=0$. Thus with
respect to the decomposition $\mathcal H=\mathcal H_1\oplus
\mathcal H_2$
\[
S_i=
\begin{bmatrix}
S_{i11}&0\\
0&S_{i22}
\end{bmatrix}\,, \quad
S_{n-i}=
\begin{bmatrix}
S_{(n-i)11}&0\\
0&S_{(n-i)22}
\end{bmatrix}.
\]
So, $\mathcal H_1$ and $\mathcal H_2$ reduce $S_1$ and $S_2$. Also
$(S_{i22},S_{(n-i)22},P_2)$, being the restriction of the $\mathbb
E$-contraction $(S_1,\dots,S_{n-1},P)$ to the reducing subspace
$\mathcal H_2$, is an $\Gamma_n$-contraction. Since $P_2$ is
completely non-unitary, $(S_{122},\dots,S_{(n-1)22},P_2)$ is a
completely non-unitary $\Gamma_n$-contraction.

\end{proof}

\vspace{0.62cm}

\noindent \textbf{Acknowledgement.} The author is thankful to the
referee for making numerous invaluable comments on the article.
The referee's suggestions helped in refining the paper.


\begin{thebibliography}{99}

\bibitem{ALY13} J. Agler, Z. A. Lykova and N. J. Young,
A case of $\mu$-synthesis as a  quadratic semidefinite program,
{\em SIAM Journal on Control and Optimization}, 2013, 51(3),
2472-2508.

\bibitem{ay-jfa} J. Agler and N. J. Young, A commutant lifting theorem for
a domain in $\mathbb{C}^2$ and spectral interpolation, \textit{J.
Funct. Anal.} 161 (1999), 452 -- 477.

\bibitem{ay-jot} J. Agler and N. J. Young, A model theory for
$\Gamma$-contractions, \textit{J. Operator Theory} 49 (2003),
45-60.

\bibitem{bhatia}  R. Bhatia, Positive definite matrices,
\textit{Princeton Series in Applied Mathematics, Princeton
University Press}, Princeton, NJ, 2007.

\bibitem{tirtha-sourav} T. Bhattacharyya, S. Pal and S. Shyam Roy,
Dilations of $\Gamma$-contractions by solving operator equations,
\textit{Adv. Math.} 230 (2012), 577 -- 606.

\bibitem{tirtha-sourav1} T. Bhattacharyya and S. Pal, A functional
model for pure $\Gamma$-contractions, \textit{J. Operator Thoery},
71 (2014), 327 -- 339.

\bibitem{BSR} S. Biswas and S. Shyam Roy, Functional
models for $\Gamma_n$-contractions and characterization of
$\Gamma_n$-isometries, \textit{J. Func. Anal.}, 266 (2014), 6224
-- 6255.

\bibitem{nagy} H. Bercovici, C. Foias, L. Kerchy and B. Sz.-Nagy,
Harmonic analysis of operators on Hilbert space, Universitext,
\textit{Springer, New York}, 2010.

\bibitem{costara1} C. Costara, On the spectral Nevanlinna-Pick
problem, \textit{Studia Math.}, 170 (2005), 23--55.

\bibitem{Gamelin} T. Gamelin, Uniform Algebras, \textit{Prentice Hall, New
Jersey}, 1969.

\bibitem{sourav} S. Pal, From Stinespring dilation to Sz.-Nagy
dilation on the symmetrized bidisc and operator models,
\textit{New York Jour. Math.}, 20 (2014), 545 -- 564.

\bibitem{sourav2} S. Pal, On decomposition of operators having $\Gamma_3$ as a spectral set,
\textit{Operators and Matrices}, 11 (2017), 891 -- 899.

\bibitem{pal-shalit} S. Pal and O. M. Shalit, Spectral sets and
distinguished varieties in the symmetrized bidisc, \textit{J.
Funct. Anal.}, 266 (2014), 5779 -- 5800.

\bibitem{vasilescu} F. H. Vasilescu, Analytic Functional
Calculus and Spectral Decompositions, \textit{Editura Academiei:
Bucuresti, Romania and D. Reidel Publishing Company}, 1982.

\bibitem{vN} J. von Neumann, Eine Spektraltheorie f\"{u}r allgemeine
Operatoren eines unit\"{a}ren Raumes, \textit{Math. Nachr.} 4
(1951), 258-281.

\end{thebibliography}
\end{document}